\documentclass[12pt]{article}

\usepackage{amsmath,amsfonts,amsthm,amssymb,graphicx}

\usepackage[mathscr]{euscript}
\usepackage{tikz}
\usepackage{multirow}
\usepackage[normalem]{ulem}

\usepackage{soul}
\usepackage{url}

\newtheorem{theorem}{Theorem}
\newtheorem{lemma}{Lemma}

\theoremstyle{definition}

\theoremstyle{remark}
\newtheorem{remark}{Remark}

\numberwithin{equation}{section}

\newcommand{\cT}{\mathcal{T}}

\newcommand{\linspan}{\operatorname{span}}
\newcommand{\norm}[1]{\|#1\|}

\newcommand{\R}{\mathbb{R}}

\newcommand{\cred}[1]{{\color{red} #1}}
\newcommand{\cblue}[1]{{\color{blue} #1}}

\newcommand{\LIU}[1]{{#1}}

\begin{document}

\title{%
Projection error-based guaranteed $L^2$ error bounds for finite element approximations of Laplace eigenfunctions 
}

\author{
Xuefeng LIU\thanks{X. Liu is supported by Japan Society for the Promotion of Science: Fund for the Promotion of Joint International Research (Fostering Joint International Research (A)) 20KK0306, Grant-in-Aid for Scientific Research (B) 20H01820, 21H00998. 
}
\\ 
\small Faculty of Science, Niigata University,\\
\small 8050 Ikarashi 2-no-cho, Nishi-ku, Niigata City, Niigata 950-2181, Japan,\\
\small email: xfliu@math.sc.niigata-u.ac.jp
\\ 
\\
Tom\'a\v s Vejchodsk\'y\thanks{T. Vejchodsk\'y is supported by the Czech Science Foundation, project no.~20-01074S, and by RVO 67985840.}\\
\small Institute of Mathematics, Czech Academy of Sciences,\\
\small \v{Z}itn\'a 25, Prague 1, 115\,67, Czech Republic,\\
\small email: vejchod@math.cas.cz
}


\maketitle




\begin{abstract}
For conforming finite element approximations of the Laplacian eigenfunctions, a fully computable guaranteed error bound in the $L^2$ norm sense is proposed.
The bound is based on the {\em a priori}  error {estimate} for the Galerkin projection of {the} conforming finite element method, and has an optimal speed of convergence for the eigenfunctions with the worst regularity. The resulting error estimate bounds the distance of spaces of exact and approximate eigenfunctions and, hence, is robust even in the case of multiple and tightly clustered eigenvalues. The accuracy of the proposed bound is illustrated by numerical examples.
\end{abstract}




\section{Introduction}

Deriving guaranteed error bounds for approximate eigenfunctions of the Laplace operator is a challenging task due to possible ill-posedness of eigenfunctions. In the case of multiple and tightly clustered eigenvalues, the corresponding exact
\LIU{eigenfunctions} are sensitive even to small perturbations of the problem and may change abruptly. Any accurate  error bound has to take into  account this sensitivity. Therefore, the recent results \cite{CanDusMadStaVoh2017,CanDusMadStaVoh2018,CanDusMadStaVoh2020,LiuVej2022} consider an arbitrary cluster of eigenvalues and the space generated by corresponding eigenfunctions. The resulting error bounds then estimate a distance between the eigenfunction spaces associated to exact and approximate eigenvalues. The particular distances between spaces are naturally based either on the energy or $L^2$ norm.

Interestingly, the $L^2$ norm bounds  provided by Algorithm I of  \cite{LiuVej2022}, which is based on the Rayleigh quotients of the approximate eigenfunctions, are considerably less accurate than the corresponding bounds in the energy norm. 
For the linear FEM solutions to the Laplacian eigenvalue problems, Algorithm I of \cite{LiuVej2022} provides the error bounds in the energy norm that exhibit the optimal rate of convergence. However, the $L^2$ error bounds often converge sub-optimally and are considerably less accurate. It is worth pointing out that, the residual error-based Algorithm II of  \cite{LiuVej2022} provides an optimal $L^2$ norm bound through the re-constructed flux and the Prager--Synge method; see the computation results for the L-shaped domain in \S \ref{sec:l-shape}.

\medskip

In this paper, we bound the $L^2$ error by utilizing the {\em a priori} error estimation 
 proposed {in} \cite{LiuOis2013} for the boundary value problems, and the resulting bound achieves the optimal rate of convergence and more accurate numerical results.
For eigenfunction space $E$ associated to eigenvalues in a cluster $\mathcal{C}$ with eigenvalues $\{\lambda_n, \cdots, \lambda_N\}$, we obtain the following bounds in Lemma \ref{thm:u_and_pi_k_u_new_ver} and Theorem \ref{th:main_new_}:
\begin{equation*}
\delta_b(E,E_h)
\le  (1+\beta)
  \max_{u \in E, \|\nabla u\|=1} \|u - P_h u\|
  \le (1+\beta) {\lambda_N} C_h^2~.
\end{equation*}
Here, $\delta_b(E,E_h)$ is the $L^2$ norm directed distance  to measure the distance between $E$ and its approximate eigenspace $E_h$; 
$\beta$ is a quantity related to the cluster width and the gap between the 
cluster $\mathcal{C}$  and the rest 
of the spectrum;
$C_h$ is a quantity  with 
an explicitly known or computable value that comes from
the {\em a priori} error estimation for the projection operator $P_h$. The proposed 
estimate
of quantity $\beta$ can be regarded as an improvement of \cite{CarGed2011}; see the comparison in Remark \ref{remark:comparison_carsten_gedick}. Also, the obtained explicit bounds are consistent with the standard  qualitative analysis for the eigenfunction approximations; see the discussion in Remark \ref{remark:convergence-rate}.

\medskip

To evaluate the bounds on eigenfunctions, suitable lower and upper bounds on eigenvalues are required. In this paper, we assume that sufficiently accurate two-sided bounds on eigenvalues are available, although we admit that computing guaranteed eigenvalue bounds, especially from below, is not a simple task. 
We use the recent method \cite{Liu2015} based on the explicitly know interpolation constant for the Crouzeix--Raviart finite element method; see also, \cite{LiuOis2013,CarGal2014,CarGed2014}. This method provides lower bounds on eigenvalues and we further use the Lehmann--Goerisch method \cite{Lehmann1949,Lehmann1950,GoeHau1985} to compute their high-precision improvements.



Let us note that there is a vast literature on error estimates for symmetric elliptic eigenvalue problems. Classical works \cite{Chatelin1983,BabOsb:1991,Boffi:2010} provide the fundamental theories. 
Many existing {\em a posteriori} error bounds on eigenvalues contain unknown constants or are valid asymptotically; see, e.g., \cite{DurGasPad1999,ArmDur2004,Yang2010,MehMie2011,DarDurPad2012,GiaHal2012,JiaCheXie2013,HuHuaLin2014}.
In the last years, several results providing fully computable and guaranteed a posteriori error estimates for eigenvalues appeared; see \cite{CarGal2014,CarGed2014,Liu2015,LiuOis2013,SebVej2014,Vejchodsky2018b,Vejchodsky2018,carstensen2021direct}.
These estimates contain no unknown constants and bound eigenvalues on all meshes, not only asymptotically.

In particular, the general framework proposed in \cite{Liu2015} was applied to the Stokes eigenvalue problem \cite{Xie2LIU-2018}, Steklov eigenvalue problem \cite{you-xie-liu-2019}, and biharmonic operators related to the quadratic interpolation error constants \cite{liu-you:2018,LiaoYuLiu2019}.
The series of papers \cite{CanDusMadStaVoh2017,CanDusMadStaVoh2018,CanDusMadStaVoh2020} provides guaranteed, robust, and optimally convergent {\em a posteriori} bounds for eigenvalues and even for corresponding eigenfunctions for both conforming and nonconforming approximations. The last paper in the series solves the difficult case of multiple and tightly clustered eigenvalues. 
The recent work \cite{LiuVej2022} proposes two algorithms to handle multiple and tightly clustered eigenvalues as well and provides alternative guaranteed and fully computable error bounds for eigenfunctions. 
Particularly, the residual error-based  Algorithm II of \cite{LiuVej2022} provides high-precision bounds by successfully extending  the Davis--Kahan theorem to weakly formulated eigenvalue problems.



\medskip

The rest of the paper is organized as follows.
Section~\ref{se:eigenproblem} briefly recalls the Laplace eigenvalue problem, its discretization by the finite element method, and division of the spectrum into clusters. 
Section~\ref{se:optimalL2} derives a project error-based bound for finite element eigenfunctions in the $L^2$ {sense}. 
Section~\ref{se:numex} presents the results of two numerical examples and Section~\ref{se:conclusions} draws the conclusions. 
Below is url of the online demonstration: 
\begin{center}

\url{https://ganjin.online/xfliu/EigenfunctionEstimation4FEM}

\end{center}

\section{Laplace eigenvalue problem}
\label{se:eigenproblem}

Let us considers the Laplace eigenvalue problem to find eigenvalues $\lambda_i \in \R$ and corresponding eigenfunctions $u_i \neq 0$ such that
\begin{equation}
\label{eq:modpro}
  -\Delta u_i = \lambda_i u_i \quad\text{in }\Omega,
  \qquad
  u_i = 0 \quad\text{on }\partial\Omega,
\end{equation}
where $\Omega \subset \R^d$ is a bounded, Lipschitz $d$-dimensional domain.
The weak formulation of this eigenvalue problem reads: find $\lambda_i \in \R$ and $u_i \in H^1_0(\Omega) \setminus \{0\}$ such that
\begin{equation}
  \label{eq:weakf}
  (\nabla u_i, \nabla v) = \lambda_i (u_i,v) \quad \forall v \in H^1_0(\Omega),
\end{equation}
where $H^1_0(\Omega)$ is the usual Sobolev space of square integrable functions with the square integrable gradients and with zero traces on the boundary $\partial\Omega$; and $(\cdot,\cdot)$ stands for the $L^2(\Omega)$ inner product.

The Laplace eigenvalue problem is well studied in \cite{BabOsb:1991,Boffi:2010}.
There exists a countable sequence of eigenvalues
$$
  0 < \lambda_1 \leq \lambda_2 \leq \cdots,
$$
where we repeat each eigenvalue according to its multiplicity.
The corresponding eigenfunctions $u_i \in H^1_0(\Omega)$ are assumed to be normalized such that
$$
  (u_i,u_j) = \delta_{ij}, \quad i,j = 1,2, \dots.
$$


We discretize problem \eqref{eq:weakf} by the standard conforming finite element method. For simplicity, we assume $\Omega$ to be a polytope.
We consider the usual conforming simplicial mesh
$\cT_h$ in $\Omega$ and define the finite element space $V_h$ of piece-wise polynomial and continuous functions over the mesh 
$\cT_h$ satisfying the Dirichlet boundary conditions as
$$
V_h =\{v_h  \in H^1_0(\Omega)  :  {v_h} |_K \in \mathbb{P}_p(K) \text{ for all } K \in \cT_h \},
$$
where $\mathbb{P}_p(K)$ stands for the space of polynomials of degree at most $p$ defined in $K$.

The finite element eigenvalue problem reads:
find $\lambda_{h,i}\in\R$ and $u_{h,i} \in V_h\setminus\{0\}$ such that
\begin{equation}
\label{eq:eig_pro_with_fem}    
(\nabla u_{h,i,} \nabla v_h) = \lambda_{h,i} (u_{h,i}, v_h)\quad \forall v_h \in V_h,
\end{equation}
where $i=1,2,\dots,\operatorname{dim} (V_h)$. 
Discrete eigenfunctions are assumed to be normalized such that $(u_{h,i},u_{h,j})=\delta_{ij}$ and $(\nabla u_{h,i},\nabla u_{h,j})=\lambda_{h,i} \delta_{ij}$.

As we mentioned in the introduction, {we will formulate the $L^2$ error bound on eigenfunctions for clusters of eigenvalues.} For the purpose of the theory, the splitting of the spectrum into clusters can be arbitrary. 
Let $n_k$ and $N_k$ stand for indices of the first and the last eigenvalue in the $k$th cluster; see Figure~\ref{fi:clusters}. 
{Note that eigenvalues in a cluster need not equal to each other}.
We consider the $k$th cluster to be of interest, and set $n=n_k$ and $N=N_k$ 
to simplify the notation.
Let $E_k$ be the space of exact eigenfunctions associated to 
$k$th cluster of eigenvalues:
$$
E_k = \linspan\{ u_{n}, u_{n+1}, \dots, u_{N} \}
~.
$$
Similarly, finite element approximations $u_{h,i} \in H^1_0(\Omega)$ 
of exact eigenfunctions $u_i$, {for $i=n,n+1,\dots,N$}, form the corresponding approximate space:
$$
E_{h,k} = \linspan\{  u_{h,n}, u_{h,n+1}, \dots,  u_{h, N} \}~.
$$


\begin{figure}[t]
\begin{tikzpicture}[scale=1]
\newcommand{\tlen}{0.1}
\newcommand{\tick}[1]{\draw [semithick] (#1,-\tlen)--(#1,\tlen);}
\draw [semithick] (0,0)--(6.2,0);
\draw [thick,dotted] (6.2+0.2,0)--(6.2+1.3,0);
\draw [thick] (6.2+1.5,0)--(6.2+7,0);

\tick{0.5}\node [below] at (0.5,-\tlen) {$0$};

\tick{1.7}\node [above] at (1.7,\tlen) {$\lambda_{n_1}$};
\tick{1.9}
\tick{2.5}\node [above] at (2.3,\tlen) {$\cdots$};
\tick{2}
\tick{2.7}
\node [above] at (2.9,\tlen) {$\lambda_{N_1}$};
\node [below] at (2.25,-\tlen) {$1$st cluster};

\tick{4}\node [above] at (4,\tlen) {$\lambda_{n_2}$};
\tick{4.25}
\tick{4.6}\node [above] at (4.7,\tlen) {$\cdots$};
\tick{5.1}
\tick{5.2}
\node [above] at (5.3,\tlen) {$\lambda_{N_2}$};
\node [below] at (4.75,-\tlen) {$2$nd cluster};

\tick{9.5}\node [above] at (9,\tlen) {$(\lambda_n:=)\lambda_{n_k}$};
\tick{9.65}
\tick{10}
\tick{10.1}
\tick{10.35}
\tick{10.5}
\tick{10.7}
\tick{10.8}
\tick{11}\node [above] at (11.75,\tlen) {$\lambda_{N_k}(=:\lambda_{N})$};
\node [below] at (10.5,-\tlen) {$k$th cluster};
%
\end{tikzpicture}
\caption{Clusters of eigenvalues on the real axis.}
\label{fi:clusters}
\end{figure}
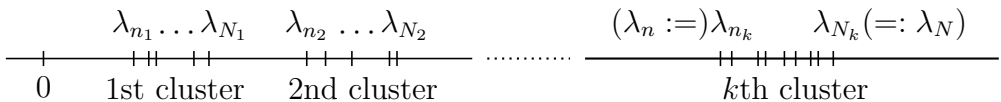

Denoting by $\|\cdot\|$ the $L^2(\Omega)$ norm, the directed distances of spaces measured in the energy and $L^2$ norms are defined as follows.
\begin{equation}
\label{eq:Delta}
\delta_a(E_k, E_{h,k}) = \max_{\substack{v \in E_k\\ \|\nabla v \|=1}} \min_{ v_h \in E_{h,k}} \|\nabla v - \nabla v_h \|
,\quad
\delta_b(E_k, E_{h,k}) = \max_{\substack{v \in E_k\\ \| v \|=1}} \min_{ v_h \in E_{h,k}} \| v- v_h \|.
\end{equation}


For reader's convenience and for the later reference, we recall the recent error bounds from \cite{LiuVej2022}. 
Take $\rho$ such that 
$\lambda_n < \rho \leq \lambda_{N+1}$, then
{
\begin{align}
  \label{eq:Deltaest}
  \delta_a^2(E_k,E_{h,k})
  &\leq  {\frac{\rho (\hat\lambda^{(k)}_N -  \lambda_n)  + \lambda_n \hat\lambda^{(k)}_N \vartheta^{(k)}}{\hat\lambda^{(k)}_N(\rho - \lambda_n)}}
   \quad\text{and}
\\
  \label{eq:deltaest}
  \delta_b^2(E_k,E_{h,k})
  &\leq {\frac{\hat\lambda^{(k)}_N - \lambda_n + \theta^{(k)}}{\rho - \lambda_n}},
\end{align}
where
\begin{align*}
  \hat\lambda^{(k)}_N &= \max_{v_h \in E_{h,k}} \frac{\norm{\nabla v_h}^2}{\norm{v_h}^2},
\quad
  \vartheta^{(k)} = \sum_{\ell=1}^{k-1} \left( \frac{\rho}{ \lambda_{n_\ell} } - 1 \right) \left[ \hat{\zeta}(E_{h,\ell},E_{h,k}) + \delta_a(E_\ell,E_{h,\ell}) \right]^2, 
\\
  \theta^{(k)} &= \sum_{\ell=1}^{k-1} \left(\rho - \lambda_{n_\ell}\right) \left[ \hat{\varepsilon}(E_{h,\ell},E_{h,k}) + \delta_b(E_\ell,E_{h,\ell}) \right]^2.
\end{align*}
Note that quantities
\begin{equation*}
  \hat{\zeta}(E_{h,\ell},E_{h,k}) = \max_{\substack{v\in E_{h,\ell}\\ \|\nabla v\|=1}} \max_{\substack{w\in E_{h,k}\\ \|\nabla w\|=1}} (\nabla v, \nabla w),
\quad
  \hat{\varepsilon}(E_{h,\ell},E_{h,k}) = \max_{\substack{v\in E_{h,\ell}\\ \| v\|=1}} \max_{\substack{w\in E_{h,k}\\ \|w\|=1}} ( v , w )
\end{equation*}
}%
%
measure the non-orthogonality between spaces of approximate eigenfunctions for the previous clusters and can be easily computed by using \cite[Lemma~2]{LiuVej2022}.
{Further}
{note that in \cite{LiuVej2022}, the approximate eigenfunction $\{u_{h,i}\}$ 
are considered as 
arbitrary and the orthogonality of $\{u_{h,i}\}$ is not required.}

\section{Projection error-based estimate in the $L^2$ norm}
\label{se:optimalL2}

The result of \cite[Theorem 8.1]{Boffi:2010}, and the explicitly known value of the constant in the \emph{a priori} error estimate for the energy projection \cite{LiuOis2013} enable us to mimic this approach for the eigenvalue problem and derive an optimal order convergent guaranteed and fully computable upper bound on the directed distance of the exact and approximate spaces of eigenfunctions measured in the $L^2(\Omega)$-norm.

First, we mention that $u_{h,i}$ is not available in practical computation, in general, because it is a result of a
generalized matrix eigenvalue solver polluted typically by rounding errors and truncation errors of iterative algorithms. 
In principle, 
{we could apply the interval arithmetic to have a rigorous representation of $u_{h,i}$,}
{but such} argument would make the paper lengthy and not easy to read. Therefore, we concentrate here on a theoretical analysis of the discretization error $(u_{h,i}-u_i)$,
{where $u_{h,i}$ is the exact solution of the discrete problem \eqref{eq:eig_pro_with_fem}.}

For the reader's convenience, we recall several results about the \emph{a priori} error estimates for finite element solutions of the Poisson equation. These \emph{a priori} error estimates will play an important role in subsequent error bounds for eigenfunctions.

Given $f\in L^2(\Omega)$, let $u\in H_0^1(\Omega)$ be the weak solution of the Poisson problem satisfying
$$
(\nabla u, \nabla v) = (f,v) \quad \forall v \in H^1_0(\Omega).
$$
The corresponding Galerkin approximation $u_h \in V_h(\subset H^1_0(\Omega))$ is determined by the identity
$$
(\nabla u_h, \nabla v_h) = (f,v_h) \quad \forall v_h \in V_h.
$$
The energy projector $P_h : H_0^1(\Omega) \rightarrow V_h$ is defined by 
$(\nabla (u - P_h u), \nabla v_h) = 0$ for all $v_h \in V_h$. Clearly,  $u_h = P_h u$.

In \cite{LiuOis2013}, Liu proposed the following constructive \emph{a priori} error estimate with a computable constant $C_h$:
\begin{equation}
\label{eq:Ch}
\|\nabla(u-P_h u) \| \le C_h \|f\|, \quad \| u - P_h u \| \le C_h \| \nabla(u-P_h u) \| \le C_h^2 \|f\|
\end{equation}
and the following lower eigenvalue bounds:
\begin{equation}
\lambda_{k} \ge \frac{\lambda_{h,k}}{1+C_h^2 \lambda_{h,k}}~\quad (k=1, 2,\cdots, \operatorname{dim}(V_h)).    
\end{equation}
%
In case of non-convex domains, the value of $C_h$ can be computed by solving a dual saddle-point problem based on the hypercircle method; see \cite[Sections 3.2--3.3]{LiuOis2013}. 
In case of convex domains, the value of $C_h$ can be easily computed by considering the Lagrange interpolation error constant; see \cite[Theorem~3.1]{LiuOis2013}. 
The specific value of $C_h$ is provided below in Section~\ref{se:numex} for the considered examples.

Throughout this section, we consider an arbitrary cluster of eigenvalues $\lambda_n, \lambda_{n+1}, \dots, \lambda_N$.
We denote by 
$\mathcal{C} = \{n, n + 1, \dots, N\}$
the set of indices of eigenvalues in this cluster
and by $|\mathcal{C}| = N - n + 1$ their number.
Spaces of exact and finite element eigenfunctions corresponding to this clusters are 
$E = \linspan\{ u_i : i \in \mathcal{C} \}$ and
$E_h = \linspan\{ u_{h,i} : i \in \mathcal{C} \}$, respectively.

It is also assumed that 
\begin{equation}    \label{eq:no_overlap_of_eigs}
\lambda_{h,n-1} < \lambda_n, \quad \lambda_N < \lambda_{h,N+1}.
\end{equation}
Such an assumption makes it possible to define the following quantities:
\begin{equation*}
\tau = \max_{j \in \mathcal{C}}\max_{i {\in}\mathcal{I} \setminus \mathcal{C}}\frac{\lambda_j}{|\lambda_{h,i}-\lambda_j|}
,\quad
\tau_h = \max_{j \in \mathcal{C}}\max_{i {\in}\mathcal{I} \setminus \mathcal{C}}\frac{\lambda_{h,i}}{|\lambda_{h,i}-\lambda_j|},
\end{equation*}
where $\mathcal{I} = \{1,2,\dots,\operatorname{dim} (V_h)\}$ stands for the set of all indices.
These quantities extend the one in \cite[pages 53, 57]{Boffi:2010} and have their origin in \cite{raviart1983introduction}. 
The application of the quantity $\tau $ can be found in \cite[Prop.~3.1]{CarGed2011}. The result in  Lemma~\ref{thm:u_and_pi_k_u} can be regarded an improvement of the one of \cite{CarGed2011}.

To derive 
the projection error-based 
upper bound 
on the directed distance of the exact and approximate spaces of eigenfunctions measured in the $L^2(\Omega)$-norm
by applying estimates \eqref{eq:Ch}, we need to
bound the error of the $L^2(\Omega)$ orthogonal projection $\Pi^\mathcal{C}_h: H^1_0(\Omega) \rightarrow E_h$ 
by the error of the energy projection $P_h: H^1_0(\Omega) \rightarrow V_h$. To achieve this goal, we first introduce several quantities and two auxiliary lemmas. 


\medskip


Let us introduce the unit ball $E^B := \{ u \in E : \|u\|=1 \}$.
For the given cluster of eigenfunction, we introduce $\beta$ 
as the optimal (minimal) quantity  that makes the 
inequality 
\begin{equation}
 \label{def:beta}
 \|(I-\Pi^\mathcal{C}_h) P_h u\| \le 
\beta 
  \max_{\substack{v \in E^B}} \|v - P_h v\| \quad 
  \mbox{hold for any } u\in E^B
 ~,
\end{equation}
and aim to obtain an upper bound of $\beta$.
In case $\|u - P_h u\|=0$ for all $u\in E^B$, it is natural to define $\beta=0$. 
Given $u =\sum_{j\in \mathcal{C}} c_j u_j \in E^B$,
let $\kappa:E^B \to E^B$ be the mapping such that
\begin{equation}
	\label{eq:def-kappa-h}
\kappa u = 
\overline{\lambda}^{-1}
\sum_{j\in \mathcal{C}} {c_j \lambda_j}  u_j,\quad 
\text{where}\quad
{\overline{\lambda}^2} =\sum_{j\in \mathcal{C}}c_j^2 \lambda_j^2 .
\end{equation}
It is easy to see that $\kappa:E^B \to E^B$ is bijective.
We set $\overline u = \kappa u$, define the relative width $\xi := (\lambda_N - \lambda_n)/\lambda_n$ of the eigenvalue cluster of interest, and note that the following estimate holds:
$$
\|u-\overline{u}\|^2 = \sum_{j\in \mathcal{C} } c_j^2 (1-\lambda_j/\overline{\lambda})^2 \le \xi^2. 
$$

\begin{lemma}\label{thm:u_and_pi_k_u}
Given an arbitrary clusters of eigenvalues, 
the quantity $\beta$ satisfies
\begin{equation}
    \label{eq:est_of_beta_k}
    \beta \leq \tau  |\mathcal{C}|^{1/2}.
\end{equation}
Further, if 
\begin{equation}
    \label{eq:condition_for_lemma}
\tau_h\xi < 1 - |\mathcal{C}|^{-1/2}
\end{equation}
then 
\begin{equation}
    \label{eq:est_of_beta_k_sharper}
    \beta \leq \frac{\tau}{1-\tau_h\xi} ~.
\end{equation}

\end{lemma}
\begin{remark}
Note that if condition \eqref{eq:condition_for_lemma} is satisfied than the estimate \eqref{eq:est_of_beta_k_sharper} is always sharper then the bound \eqref{eq:est_of_beta_k}.
Further, a smaller relative width of the cluster $\xi$  leads to a sharper upper bound \eqref{eq:est_of_beta_k_sharper}.
 In the extreme case of a multiple eigenvalue such that $\lambda_n=\lambda_N$, we have $\xi = 0$ and hence $\beta \le \tau$.
\end{remark}

\begin{proof}
{First, for 
 $u_j \in E^B$ as an eigenfunction, let us apply the standard argument (see, e.g., \cite{Boffi:2010}) to show that
 $\|(I - \Pi^\mathcal{C}_h) P_h u_j\| \le \tau \| (I - P_h) u_j\|$. }
 Note that
$$
(I - \Pi^\mathcal{C}_h) P_h u_j = \sum_{i \in \mathcal{I} \setminus \mathcal{C}} (P_h u_j, u_{h,i}) u_{h,i} \in V_h,
$$
which leads to 
\begin{equation}
\label{eq:IPik}
\|(I  - \Pi^\mathcal{C}_h ) P_h u_j\|^2 = \sum_{i \in \mathcal{I} \setminus \mathcal{C}} (P_h u_j, u_{h,i})^2\:.
\end{equation}
In equality
\begin{equation}
    \label{eq:equality_in_lemma}
\lambda_{h,i} (P_h u_j , u_{h,i}) = (\nabla P_h u_j, \nabla u_{h,i})
= (\nabla u_j, \nabla u_{h,i}) = \lambda_j (u_j,u_{h,i}),
\end{equation}
we subtract $\lambda_j (P_h u_j,u_{h,i})$ on both sides and obtain
$$
(P_h u_j, u_{h,i}) = \frac{\lambda_j}{\lambda_{h,i} -\lambda_j } (u_j - P_h u_j,u_{h,i}).
$$
Summation over $i \in \mathcal{I}\setminus\mathcal{C}$ gives
$$
\sum_{i \in \mathcal{I} \setminus \mathcal{C}} (P_h u_j, u_{h,i})^2 \le \tau^2 \sum_{i \in \mathcal{I} \setminus \mathcal{C}} (u_j - P_h u_j,u_{h,i})^2 \le \tau^2 \|u_j - P_h u_j\|^2,
$$
where the last inequality follows form the identity $\sum_{i \in \mathcal{I}} (u_j - P_h u_j,u_{h,i})^2 = \| \Pi_h(u_j - P_h u_j) \|^2$ with $\Pi_h : H^1_0(\Omega) \rightarrow V_h$ denoting the $L^2(\Omega)$ orthogonal projector.
Using this in \eqref{eq:IPik}, we finally derive 
\begin{equation}
\label{eq:est_for_single_uj}
\|(I - \Pi^\mathcal{C}_h) P_h u_j\| \le \tau \| (I - P_h) u_j\|.
\end{equation}


Next, 
we consider any $u \in E^B$ and express it in the form $u =\sum_{j\in \mathcal{C}} c_j u_j$ with $\sum_{j\in\mathcal{C}} c_j^2 = 1$. 
Denoting the linear operator $(I - \Pi^\mathcal{C}_h) P_h $ by $L$,
the estimate \eqref{eq:est_for_single_uj} leads to
$$
\|L u \|^2 = \left\| \sum_{j \in \mathcal{C}}  c_j L u_j \right\|^2 \le \sum_{j \in \mathcal{C}} \|Lu_j\|^2    \le \tau^2 \sum_{j \in \mathcal{C}}  \| (I - P_h) u_j\|^2.
$$
Thus, we can estimate $\|(I - \Pi^\mathcal{C}_h) P_h u\|$ as
\begin{equation}
\label{eq:lemma_estimate_for_general_u_ver0}
\|(I - \Pi^\mathcal{C}_h) P_h u \| \le  
 \tau  |\mathcal{C}|^{1/2} \max_{u \in E^B} \|u - P_h u\|
\end{equation}
and statement \eqref{eq:est_of_beta_k} easily follows.

Finally, we consider the case when the condition \eqref{eq:condition_for_lemma} holds true. 
Given $u\in E^B$, we take $\overline{u}=\kappa u$ and $\overline{\lambda}$ as defined in \eqref{eq:def-kappa-h}.
%
Using inequality \eqref{eq:equality_in_lemma}, we obtain for $u$ the identity
$$
\lambda_{h,i} (P_h u , u_{h,i}) 
= \overline{\lambda} (\overline{u},u_{h,i}).
$$
Subtracting $\overline{\lambda} (P_h \overline{u},u_{h,i})$ on both sides, we derive
$$
\left( \lambda_{h,i} P_h (u - \overline{u}) 
+ (\lambda_{h,i}
 -\overline{\lambda}) P_h \overline{u} , u_{h,i} \right) = \overline{\lambda} (\overline{u} -P_h \overline{u} ,u_{h,i}).
$$
Thus, 
$$
(P_h \overline{u}, u_{h,i}) = \frac{\overline{\lambda}}{\lambda_{h,i} -\overline{\lambda} } ( \overline{u} - P_h \overline{u},u_{h,i})
- 
\frac{\lambda_{h,i}}{\lambda_{h,i} -\overline{\lambda} } 
(P_h (u - \overline{u}), u_{h,i}).
$$
Since function $g(t)=t/ |\lambda_{h,i}-t|$ satisfies
$g(t) \le \max(g(\lambda_n), g(\lambda_N))$ for all $t \in [\lambda_n, \lambda_N]$
and function $g_h(t)=\lambda_{h,i}/|\lambda_{h,i}-t|$ is bounded in the same way, we have
$$
|(P_h \overline{u}, u_{h,i})| \leq \tau |\left(( I - P_h) \overline{u},u_{h,i}\right)| + \tau_h
|(P_h (u - \overline{u}), u_{h,i})|.
$$
Now, considering these inequalities for $i\in\mathcal{I}\setminus\mathcal{C}$, using the geometric inequality\footnote{
Given  vectors $a=(a_1, \cdots, a_n), b=(b_1, \cdots, b_n), c=(c_1, \cdots, c_n)$ with $a_i,b_i,c_i>0$ and
$a_i\leq b_i+c_i$, then their Euclidean norms satisfy
$\|a\| \le \|b\|+\|c\|.$
}
and the general fact that 
$\sum_{i\in \mathcal{I}\setminus\mathcal{C}} (\varphi,u_{h,i})^2
= \| (I - \Pi^\mathcal{C}_h) \Pi_h \varphi \|^2$ for any $\varphi \in H^1_0(\Omega)$ with $\Pi_h: H^1_0(\Omega) \rightarrow V_h$ being the $L^2(\Omega)$ orthogonal projector,  
we derive the bound
\begin{equation}
\label{eq:sub_inequality_1}    
\|(I-\Pi^\mathcal{C}_h)P_h {\overline{u}}\| \le 
\tau \|\overline{u}-P_h \overline{u}\| + 
\tau_h
\|(I-\Pi^\mathcal{C}_h)P_h (u - \overline{u})\|.
\end{equation}
%
Since $\|u-\overline{u}\| \le \xi $, the definition of $\beta$ gives
\begin{equation}
    \label{eq:sub_inequality_2}
\|(I-\Pi^\mathcal{C}_h) P_h (u-\overline{u})\|
\le 
\xi \beta \max_{u \in E^B} \|u - P_h u\|.
\end{equation}
Inequalities \eqref{eq:sub_inequality_1} and  \eqref{eq:sub_inequality_2} lead to the relation
\begin{equation}
\label{eq:lemma_estimate_for_general_u_2}
\|(I - \Pi^\mathcal{C}_h) P_h \overline{u} \| \le  
\left( \tau   + \tau_h \xi  \beta \right)
 \max_{u \in E^B} \|u - P_h u\|.
\end{equation}
Since $\overline u = \kappa u$ and $\kappa: E^B \to E^B$ is a bijection, we have
$$
  \max_{u\in E^B} \|(I - \Pi^\mathcal{C}_h) P_h \kappa u \| 
= \max_{u\in E^B} \|(I - \Pi^\mathcal{C}_h) P_h u \|.
$$
Consequently, 
from the bound \eqref{eq:lemma_estimate_for_general_u_2} and the definition of $\beta$, we obtain
$$
\beta \le  \tau   + \tau_h \xi \beta.
$$
%
Since condition \eqref{eq:condition_for_lemma} implies $1-\tau_h\xi>0$, the estimate \eqref{eq:est_of_beta_k_sharper} follows.
\end{proof}

In next lemma, we show the relation between $\delta_b(E,E_h)$ and the projection error using the quantity $\beta$. 

\begin{lemma}\label{thm:u_and_pi_k_u_new_ver}
For the given cluster of eigenvalues, 
the following estimate holds:
\begin{equation}
\label{eq:thm_u_and_pi_k_u_new_}
\delta_b(E,E_h) = \max_{u\in E^B} \|u-\Pi^\mathcal{C}_h u\| \le 
(1+\beta)
  \max_{u \in E^B} \|u - P_h u\|.
 \end{equation}
\end{lemma}
\begin{proof}
For any $u\in E^B$, since $\Pi^\mathcal{C}_h u$ provides the best approximation of $u$ under the $L^2$ norm in $E_h$, we have
\begin{equation}
\label{eq:lem_triangle_inequality_new_}
\|u - \Pi^\mathcal{C}_h u \| 
 \le \|u - \Pi^\mathcal{C}_h P_h u \| 
 \le \|u - P_h u \| + \|P_h u-\Pi^\mathcal{C}_h P_h u  \| .
\end{equation}
Using the definition of the quantity $\beta$, we easily draw the conclusion.
\end{proof}
\begin{remark}
\label{remark:comparison_carsten_gedick}
In Proposition 3.2 of \cite{CarGed2011}, with $m:=|\mathcal{C}|=N-n+1$, the following result is obtained: For $u_{h,j} $ in $E$ as an eigenfunction associated to $\lambda_{h,j}$, 
$$
\min_{u\in E} \|u_{h,j}-u\| \le 
\sqrt{2} m(2m+1)(1+\tau)
 \max_{u_j\in E^B} \|u_j-P_h u_j\|~.
$$
If such a result is applied to  $\delta_b(E,E_h)$, one can obtain the following estimate.
\begin{equation}
\label{eq:overestimation}    
\delta_b(E,E_h) \le \sqrt{2} m(2m+1)(1+\tau)   \max_{u \in E^B} \|u - P_h u\|~.
\end{equation}
This bound is larger than the result in Lemma~\ref{thm:u_and_pi_k_u_new_ver}.
Particularly, if
the eigenvalue cluster is tight, i.e., $\xi\approx 0$, we have $\beta\approx \tau$ and the bound 
\eqref{eq:overestimation} 
is overestimated by the factor $\sqrt{2} m(2m+1)$.
\end{remark}

Bounding $\beta$ by Lemma~\ref{thm:u_and_pi_k_u}, the following theorem presents
the main result.

\begin{theorem}\label{th:main_new_}
Let $\beta$ be the quantity defined in \eqref{def:beta}.
For an arbitrary cluster of eigenvalues, 
the following estimate holds:
\begin{equation}
\label{eq:l2_norm_optimal}
\delta_b(E,E_h)  \le (1+\beta) {\lambda_N} C_h^2~.
\end{equation}
\end{theorem}

\begin{proof}
Given $u\in E_k$, $u = \sum_{j\in \mathcal{C}}c_j u_j$, let $\overline{u}=\kappa u$ and $\overline{\lambda}^2 = \sum_{j\in \mathcal{C}}c_j^2 \lambda_j^2$.
Then the identity 
$$
(\nabla u, \nabla v)= (\overline{\lambda}\overline{u}, v), \quad \forall v\in V,
$$
holds and the {\em a priori} error estimate \eqref{eq:Ch} with $f=\overline{\lambda}\overline{u}$ yields
\begin{equation}
    \label{eq:local_est_1_new}
    \|u-P_h u\| \le C_h^2 \| \overline{\lambda}\overline{u}\| \le C_h^2\lambda_N .
\end{equation}
Definition of $\delta_b(E,E_h)$ and bounds \eqref{eq:thm_u_and_pi_k_u_new_} and \eqref{eq:local_est_1_new} give
\begin{equation*}
    \label{eq:est2_new}
\delta_b(E,E_h) = \max_{u\in E^B} \|u-\Pi^\mathcal{C}_h u\|
  \leq (1+\beta) \max_{u \in E^B} \| u - P_h u \| \le (1+\beta) C_h^2 \lambda_N.
\end{equation*}
\end{proof}

\begin{remark}
\label{remark:convergence-rate}
The result \cite{LiuOis2013} shows how to compute the quantity $C_h$. For convex domains, the solution of the Poisson problem has the regularity $u\in H^2(\Omega)$ and, consequently, we have
$C_h=O(h)$ via the Lagrange interpolation error estimation. 
For non-convex domains, the solution $u$ belongs to 
$H^{1+\alpha}(\Omega)$, where $\alpha \in (0,1]$ depends on the angles of re-entrant non-convex corners. 
In this case, the value of $C_h$ is evaluated by the hypercircle method using the Raviart--Thomas FEM, and it is expected that 
$C_h=O(h^{\alpha})$.

 As {it} is pointed out in \cite[Theorem 9.13]{Boffi:2010}, the FEM solutions approximate the eigenfunction independently. 
 That is, even for non-convex domains, if an 
 eigenfunction has the $H^2$-regularity, then the FEM approximation to such an eigenfunction has the $O(h^2)$ convergence rate under $L^2$ norm.
Since the projection error in the estimation \eqref{eq:thm_u_and_pi_k_u_new_} is restricted to the function in $E$ for the specified eigenvalue cluster, the estimation of Lemma \ref{thm:u_and_pi_k_u_new_ver} is consistent with the theoretical analysis of \cite{Boffi:2010}.

The proposed estimation \eqref{eq:l2_norm_optimal} using $C_h$ has a defect that, in case of non-convex domains, for an eigenfunction with a better regularity, the proposed bound still keeps the degenerated convergence rate, which is because the {\em a priori} error estimation is considering the worst case for the projection error.  
 If the regularity for eigenfunction in $E$ is known, then the estimation in Theorem \ref{th:main_new_} can also be improved since the estimation only depends on the projection error for eigenfunctions in the specified cluster.
For example, in the case of an L-shaped domain of \S \ref{sec:l-shape},  
 the eigenfunction $u=\sin(\pi x)\sin (\pi y)$ associated to  $\lambda_3=2\pi^2$ has the 
$H^2$-regularity, thus one can take $C_h<0.493h$ (where $h$ is the largest leg length for right triangles in the triangulation) for FEM approximation using triangulation with right triangles. 
\end{remark}

%





\begin{remark}
Theorem 3 in \cite{LiuVej2022} provides the following estimate:
\begin{equation}
  \label{eq:energy_by_L2}
  \delta_a^2(E,E_h) \leq {2 - 2 \lambda_n \left( \frac{1 - \delta_b^2(E,E_h) }{\lambda_N \lambda_{h,N}} \right)^{1/2}} ~,
\end{equation}
where the energy error $\delta_a(E,E_h)$ {is bounded by} the $L^2$ error $\delta_b(E,E_h)$. 
However, bound \eqref{eq:energy_by_L2} is not optimal for clusters of a positive widths, i.e., $\lambda_n<\lambda_{N}$. On the other hand, for clusters consisting of a simple or a multiple eigenvalue, we have $\lambda_n = \lambda_N$ and the bound \eqref{eq:energy_by_L2} has the optimal speed of convergence. Indeed, in this case, it can be easily shown that the right-hand side of \eqref{eq:energy_by_L2} is dominated by $|\lambda_{h,N} - \lambda_N|$ and the other terms, including $\delta_b^2(E,E_h)$ are of higher order.
{Consequently, bound \eqref{eq:energy_by_L2} combined with \eqref{eq:l2_norm_optimal} provides a guaranteed and fully computable error bound in the energy norm with the optimal speed of convergence for a cluster consisting of only one simple or multiple eigenvalue.}
\end{remark}

\section{Numerical examples}
\label{se:numex}

This section provides numerical examples to  illustrate the accuracy of proposed bounds on the directed 
distances of spaces of exact and approximate eigenfunctions. The first example is the Laplace eigenvalue problem \eqref{eq:modpro} in the unit square domain for which 
{the exact eigenvalues and eigenfunctions are well known.} 
The second example is the same problem considered 
{in a non-convex L-shaped domain where eigenfunctions may have singularities at the re-entrant corner.}

Both examples are computed in 
the floating point arithmetic and the influence of rounding errors is not taken into account
{for simplicity.} 
However, if needed, mathematically rigorous estimates could be obtained by employing the interval arithmetic \cite{moore2009introduction}.

\subsection{The unit square domain}
\label{sec:unit_square}
Consider the Laplace eigenvalue problem with homogeneous Dirichlet boundary conditions in the unit square $\Omega=(0,1)^2$: 
find eigenvalues $\lambda_i \in \R$ and corresponding eigenfunctions $u_i \neq 0$ such that
\begin{equation}
\label{eq:modpro}
  -\Delta u_i = \lambda_i u_i \quad\text{in }\Omega;
  \qquad
  u_i = 0 \quad\text{on }\partial\Omega.
\end{equation}

\begin{table}[ht]
\caption{\label{ta:square_domain_clusters}The four leading clusters for the unit square.}
\begin{center}
\begin{tabular}{|c|c|c|c|c|}
\hline
\rule[-2mm]{0cm}{6mm}{}
 Cluster & 1 & 2 & 3 & 4 \\
 \hline
\rule[-2mm]{0cm}{6mm}{}
Eigenvalues & $\lambda_1 = 2\pi^2$ &
 $\lambda_2 = \lambda_3 = 5\pi^2$  &
 $\lambda_4 = 8\pi^2$   & 
 $\lambda_5 = \lambda_6 = 10\pi^2$  \\ 
\hline
\end{tabular}
\end{center}
\end{table}

The exact eigenpairs are known analytically to be
$$
\lambda_{ij} = (i^2+j^2)\pi^2,\quad u_{ij}=\sin(i\pi x) \sin(j\pi y), \quad i,j=1,2,3, \dots.
$$
These eigenvalues are either simple or multiple
and we cluster them according to the multiplicity.
The first four clusters are listed in Table~\ref{ta:square_domain_clusters}.
Since the exact eigenvalues are known, we use them to evaluate error bounds. 
{To compute bounds \eqref{eq:Deltaest} and \eqref{eq:deltaest} for
the cluster $\{\lambda_{n}, \lambda_{n+1}, \cdots, \lambda_{N}\}$, 
we choose $\rho = \lambda_{N+1}$.
}

\begin{figure}[ht]
\begin{center}
    \includegraphics[scale=0.25]{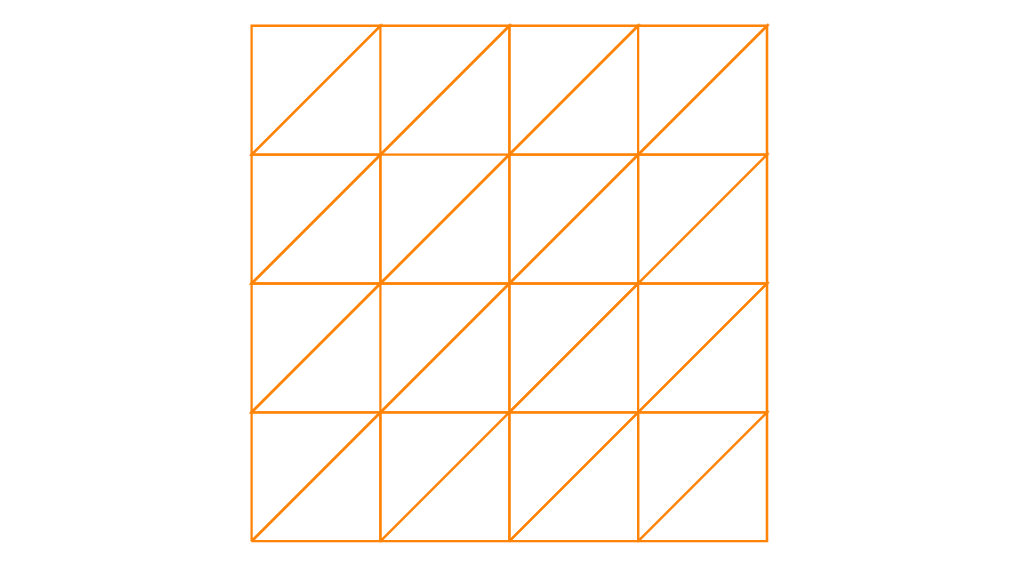}
\end{center}
\caption{\label{fig:uniform_mesh_square} The uniform mesh with $h=1/4$ for the unit square.%
}
\label{fi:squaremesh}
\end{figure}

{Problem \eqref{eq:modpro}} is discretized by the conforming finite element method using piecewise linear functions. The finite element mesh $\cT_h$ is chosen as the uniform triangulation consisting of
isosceles right triangles; 
{see Figure~\ref{fi:squaremesh} for an illustration.
}
The projection error constant can be easily obtained through the interpolation error constant
{as}
$C_h\le h/0.493$.

{
For each cluster, we compute bounds on $\delta_b(E_k,E_{h,k})$ and $\delta_a(E_k,E_{h,k})$ by 
the estimate \eqref{eq:l2_norm_optimal} from Theorem~\ref{th:main_new_} and its combination with the relation \eqref{eq:energy_by_L2}, respectively.
We then compare these results with the bounds \eqref{eq:Deltaest} and \eqref{eq:deltaest} computed by Algorithm I of \cite{LiuVej2022}.
}

The convergence behavior of computed bounds for the four leading clusters is shown in Figure ~\ref{fig:unit-square-l2} and \ref{fig:unit-square-h1}.
The results confirm the expected optimal convergence rate $O(h^2)$ of the estimate \eqref{eq:l2_norm_optimal}
for {$\delta_b(E_k,E_{h,k})$}, and the sub-optimal rate $O(h)$ from Algorithm I of \cite{LiuVej2022}. 

The estimate by Algorithm II of \cite{LiuVej2022} can provide impressively sharp bounds and the optimal convergence rate for the error of approximate eigenfunctions under both $L^2$ and $H^1$ norms. Since such an approach needs more {effort to post-process the approximate eigenfunction, reconstruct the flux, and estimate the} residual error of the eigenfunction approximation, the comparison with Algorithm II of \cite{LiuVej2022} is omitted here.

\begin{figure}[htp]
    \centering
    \includegraphics[width=\textwidth]{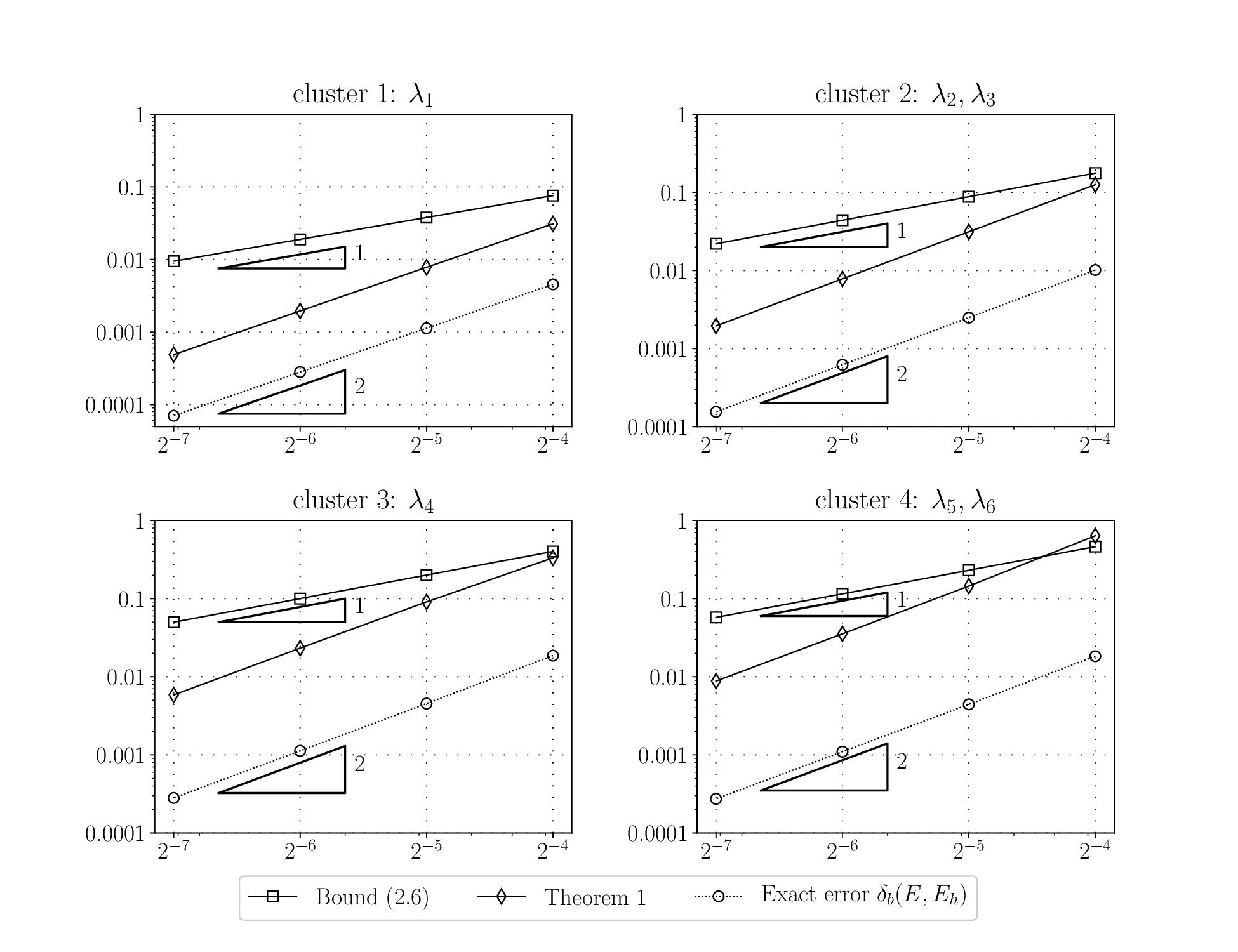}
    \caption{\label{fig:unit-square-l2}Bounds on the $L^2(\Omega)$ distances of spaces of eigenfunctions
    {$\delta_b(E_k,E_{h,k})$} for the square domain and four leading clusters of eigenvalues
    {$k=1,2,3,4$.}
    }
    
\end{figure}

\begin{figure}[htp]
    \centering
    \includegraphics[width=\textwidth]{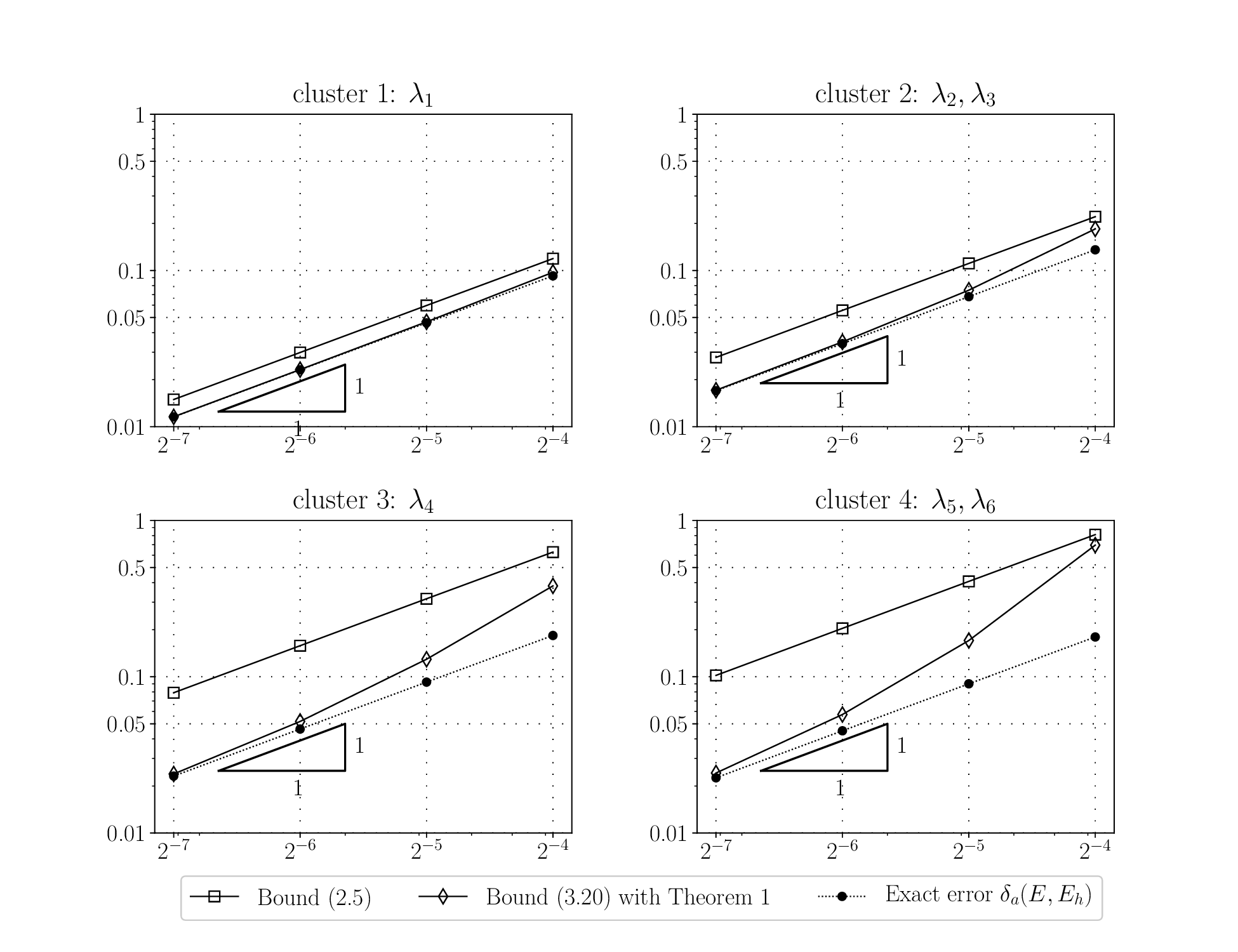}
    \caption{\label{fig:unit-square-h1}
    Bounds on the energy distances of spaces of eigenfunctions {$\delta_a(E_k,E_{h,k})$} for the square domain and the four leading clusters {$k=1,2,3,4$}.
    }
\end{figure}

\subsection{The L-shaped domain}
\label{sec:l-shape}
We consider the Laplace eigenvalue problem \eqref{eq:modpro} in the L-shaped domain $\Omega = (-1,1)^2\setminus(-1,0]^2$
to present the standard example with singularities of eigenfunctions
and also to demonstrate the versatility of the proposed method. We solve this problem by using the classical linear conforming finite element space over a uniform mesh.

\begin{figure}[ht]
    \centering
    \includegraphics[scale=0.4]{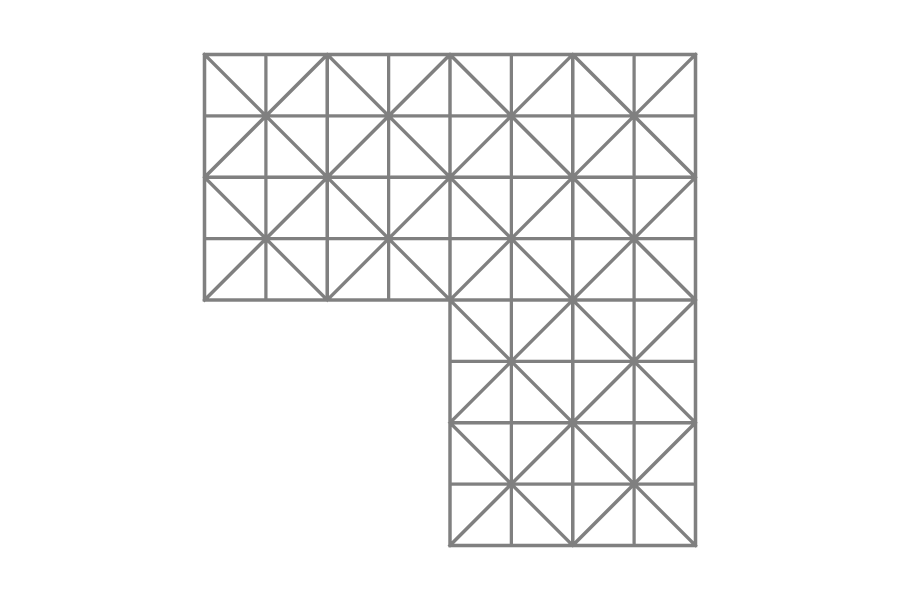}
    \caption{L-shaped domain and the initial mesh
    }
    \label{fig:l_shaped_domain}
\end{figure}

Since the exact eigenvalues are not known,
the eigenvalue bounds
are evaluated by using
two-sided bounds on eigenvalues, which were computed in \cite{liu2014high} 
and we list them in Table~\ref{tab:l_shaped_eig_lower_bound}. 
The first four eigenvalues are simple and form trivial clusters.
The values of the projection error constants are obtained by applying the hypercircle method proposed in \cite{LiuOis2013}; see Table \ref{table:lshape-projection-error-constant}.

\begin{table}[ht]
    \centering
    \caption{Lower bounds on the leading eigenvalues for the L-shaped domain.}
    \begin{tabular}{|c|c|c|c|c|}
    \hline
    \rule[-2mm]{0mm}{6mm}{}
    $\lambda_1$  & $\lambda_2$ &$\lambda_3$ & $\lambda_4$ & $\lambda_5$ \\
        \hline
    \rule[-2mm]{0mm}{6mm}{}
    $9.6397_{1}^{3}$ & $15.1972_{5}^{6}$ & $19.7392_0^1$ & $29.5214_7^9$ & $31.9126_2^4$ \\
        \hline
    \end{tabular}

    \label{tab:l_shaped_eig_lower_bound}
\end{table}

\begin{table}[h]
    \centering
    \caption{\label{table:lshape-projection-error-constant}Projection error constants}
    \begin{tabular}{|c|c|c|c|c|}
    \hline
    \rule[-2mm]{0mm}{6mm}{}
    $h$  & $1/32$ &$1/64$ & $1/128$ & $1/256$ \\
        \hline
    \rule[-2mm]{0mm}{6mm}{}
    $C_h$ & $0.0359$ & $0.0218$ & $0.0134$ & $0.00832$ \\
        \hline
    \end{tabular}

\end{table}

The initial finite element mesh is displayed in Figure~\ref{fig:l_shaped_domain}.
First, we apply the bounds \eqref{eq:Deltaest} and \eqref{eq:deltaest} to the four leading eigenvalue clusters. Since the exact error $\delta_b$ cannot be evaluated directly, we apply the residual error-based estimation, i.e., Algorithm II of \cite{LiuVej2022} to obtain a sharp bound of $\delta_b$. Numerical evaluation of such a bound implies that $\delta_b$ has the convergence rate as $O(h^{3/2})$ for the first cluster and $O(h^{2})$ for the rest $3$ clusters. 
Figure~\ref{fig:lshape-l2} shows the bounds on the $L^2(\Omega)$ distance $\delta_b$. 
Figure~\ref{fig:lshape-h1} compares the bounds on the energy distance $\delta_a$.
The results confirm that the newly proposed estimate of $\delta_b$ based on the projection error estimate, namely the estimate \eqref{eq:l2_norm_optimal} in Theorem~\ref{th:main_new_}, provides improved convergence rates in comparison with the bound \eqref{eq:deltaest}.

\begin{figure}[htp]
    \centering
    \includegraphics[width=\textwidth]{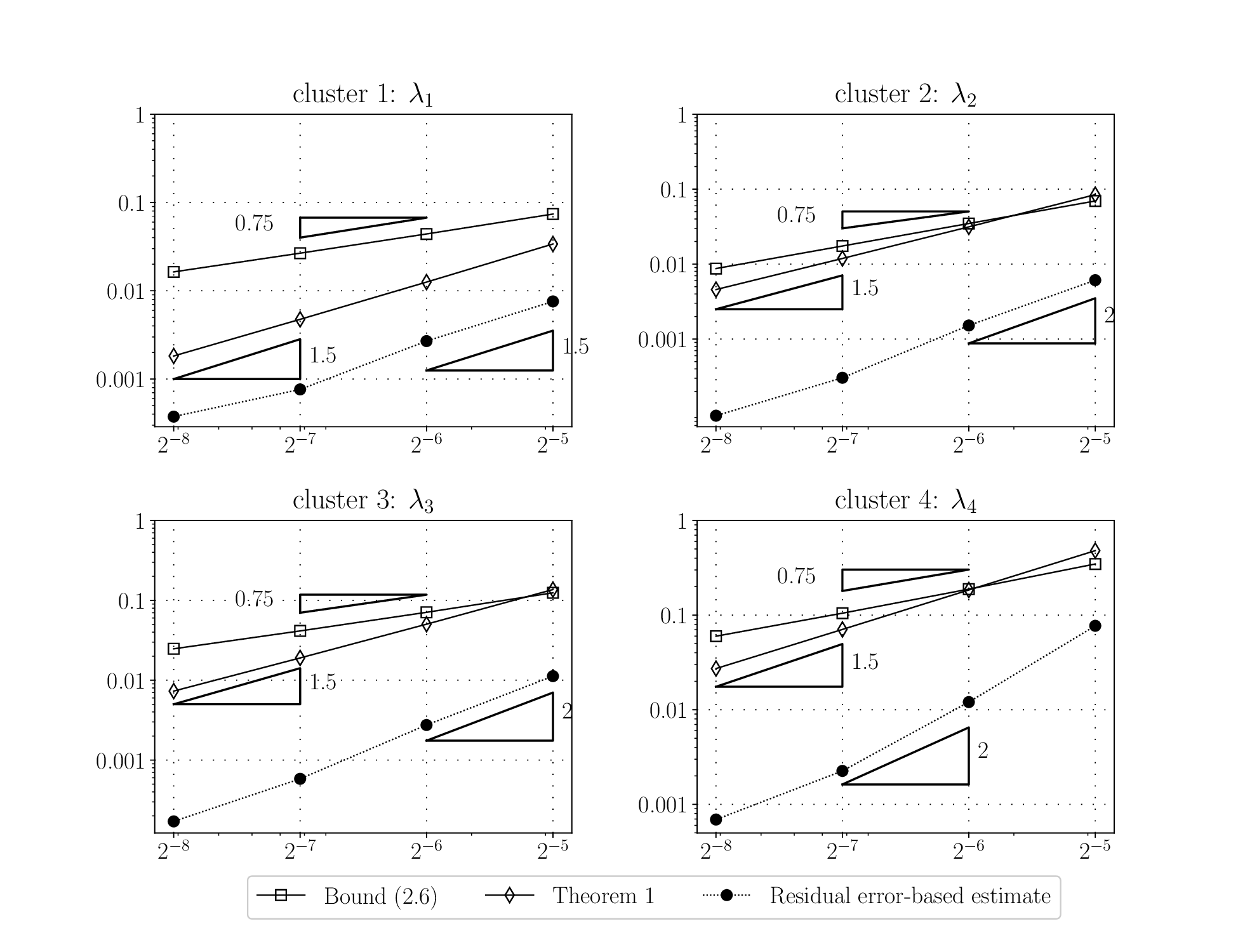}
    \caption{\label{fig:lshape-l2}Bounds on the $L^2(\Omega)$ distances of spaces of eigenfunctions
    {$\delta_b(E_k,E_{h,k})$} for the L-shaped domain and four leading clusters {$k=1,2,3,4$}.
    }
\end{figure}

\begin{figure}[htp]
    \centering
    \includegraphics[width=\textwidth]{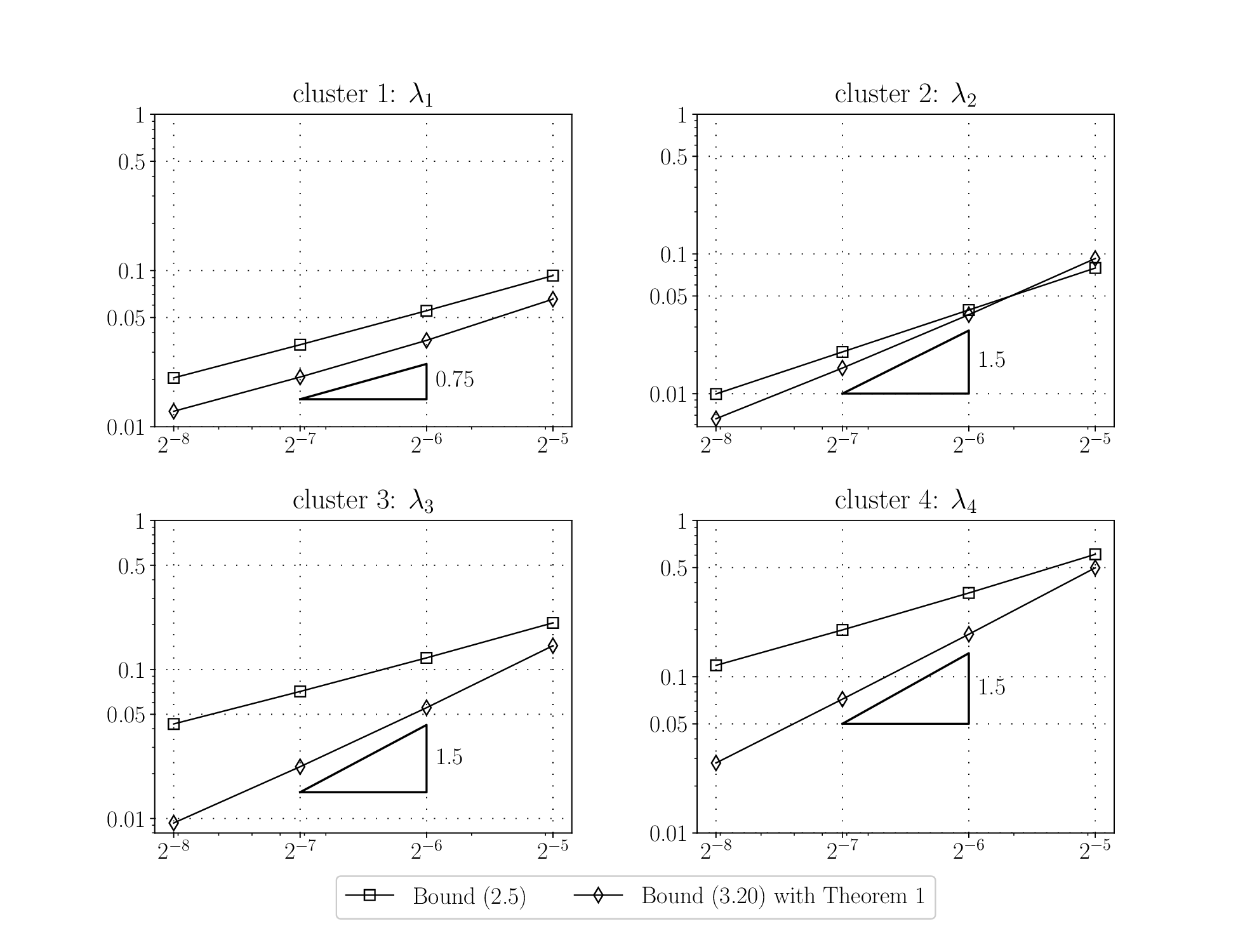}
    \caption{\label{fig:lshape-h1}Bounds on the energy distances of spaces of eigenfunctions {$\delta_a(E_k,E_{h,k})$} for the L-shaped domain and the four leading clusters {$k=1,2,3,4$}. 
    }
\end{figure}

\section{Conclusions}
\label{se:conclusions}

For finite element eigenfunctions, we derived 
a projection error-based
bound on the $L^2$ 
distance $\delta_b$ by employing the explicitly 
known value of the constant {$C_h$} in the \emph{a priori} error estimate for the energy projection.
The obtained optimal estimate of $\delta_b$ can be further utilized to improve the bound for the energy distance $\delta_a$. 
The derived bound is fully computable and guaranteed.

%

\bibliographystyle{amsplain}
\bibliography{vejchod_aee}

\end{document}